\documentclass[reqno]{amsart}
\usepackage{url}
\usepackage{amssymb,amsthm,amsfonts,amstext,amsmath}
\usepackage{pdfsync}  %%%%%pesquisa inversa pdf%%%%%%%%%%%%%%%
\usepackage{mathptmx}       % selects Times Roman as basic font
\usepackage{newcent}       % selects Times Roman as basic font
\usepackage{helvet}         % selects Helvetica as sans-serif font
\usepackage{courier}        % selects Courier as typewriter font
\usepackage{graphicx}
\usepackage{enumerate}
\usepackage{hyperref}
\numberwithin{equation}{section}
\usepackage{color}
\usepackage[active]{srcltx}%%%%%pesquisa inversa
\usepackage{pstricks}
\usepackage{float}
\usepackage{dsfont}%%%para fazer 1 (indicadora) duplo
\usepackage{mathrsfs}
\usepackage{tikz}

\newtheorem{theorem}{Theorem}[section]

\newtheorem{proposition}[theorem]{Proposition}

\newtheorem{remark}[theorem]{Remark}
\newtheorem{definition}{Definition}
\newcommand{\mc}[1]{{\mathcal #1}}

\newcommand{\bb}[1]{{\mathbb #1}}

\newcommand{\eps}{\varepsilon}

\newcommand{\p}{\partial}
\newcommand{\pfrac}[2]{\genfrac{}{}{}{1}{#1}{#2}}

\newcommand{\A}{\Delta_\beta}
\newcommand{\B}{\nabla_\beta}

\let\oldtocsection=\tocsection
\let\oldtocsubsection=\tocsubsection
\let\oldtocsubsubsection=\tocsubsubsection
\renewcommand{\tocsection}[2]{\hspace{0em}\oldtocsection{#1}{#2}}
\renewcommand{\tocsubsection}[2]{\hspace{1em}\oldtocsubsection{#1}{#2}}
\renewcommand{\tocsubsubsection}[2]{\hspace{2em}\oldtocsubsubsection{#1}{#2}}
\DeclareRobustCommand{\SkipTocEntry}[5]{}
\begin{document}

\title[Corrigendum to: equilibrium fluctuations of symmetric slowed
exclusion]{Corrigendum to: Phase transition in equilibrium fluctuations of symmetric slowed
exclusion}

\author{Tertuliano Franco}
\address{UFBA\\
 Instituto de Matem\'atica, Campus de Ondina, Av. Adhemar de Barros, S/N. CEP 40170-110\\
Salvador, Brazil}
\curraddr{}
\email{tertu@ufba.br}
\thanks{}

%    author two information
%\author{Patr\'{\i}cia Gon\c{c}alves}
%\address{\noindent Departamento de Matem\'atica, PUC-RIO, Rua Marqu\^es de S\~ao Vicente, no. 225, 22453-900, Rio de Janeiro, Rj-Brazil and CMAT, Centro de Matem\'atica da Universidade do Minho, Campus de Gualtar, 4710-057 Braga, Portugal.}
%\email{patricia@mat.puc-rio.br}

\author{Patr\'{\i}cia Gon\c{c}alves}
\address{\noindent Center for Mathematical Analysis,  Geometry and Dynamical Systems, \\
Instituto Superior T\'ecnico, Universidade de Lisboa, \\
Av. Rovisco Pais, 1049-001 Lisboa, Portugal}
\email{patricia.goncalves@math.tecnico.ulisboa.pt}

\author{Adriana Neumann}
\address{UFRGS, Instituto de Matem\'atica, Campus do Vale, Av. Bento Gon\c calves, 9500. CEP 91509-900, Porto Alegre, Brasil}
\curraddr{}
\email{aneumann@impa.br}
\thanks{}

%\subjclass[2010]{60K35, 26A24, 35K55}

\begin{abstract}
We present the correct space of test functions for the Ornstein-Uhlenbeck processes defined in  \cite{fgn2}. Under these new spaces, an invariance with respect to a second order operator is shown, granting the existence and uniqueness of those processes. Moreover, we detail how to prove some properties of the semigroups, which are required in the proof of uniqueness. 
\end{abstract}

\maketitle

\section{Outline}
In our paper \cite{fgn2} it was used that, if $H\in \mc S_\beta(\bb R)$, then 
\begin{equation}\label{eq1}
 \Delta_\beta T_t^\beta H\in \mc S_\beta(\bb R)\,,
\end{equation}
 which is not true. Above, $\Delta_\beta $ is essentially the Laplacian operator, but defined in the space of test functions $\mc S_\beta(\bb R)$ and the operator $T_t^\beta$ is the semigroup of the related PDE. The reader can find the complete definitions  in \cite{fgn2}. 

To clarify ideas, for $\beta\in[0,1)$,  
$\mc S_\beta(\bb R)$ is the classical Schwartz space, and $T_t^\beta$ is the semigroup of the heat equation. It is well known (via Fourier transform, see \cite{reedsimon} for instance) that the heat equation preserves the Schwartz space. Moreover,  any derivative of a function in the Schwartz space is again in the Schwartz space. Therefore, condition \eqref{eq1} is true for $\beta\in[0,1)$. However, for $\beta\geq 1$ it is not. In this note we properly redefine $\mc S_\beta(\bb R)$ for $\beta\geq 1$, so that \eqref{eq1} becomes true for any $\beta\geq 0$.

A remark: on top of page  \pageref{page}  of the present article we give the precise locations  where \eqref{eq1} is used in \cite{fgn2}.  One of this locations is the proof of uniqueness of the corresponding generalized Ornstein-Uhlenbeck processes (O.U.). There are some recent works on generalized O.U. processes whose associated operators do not satisfy \eqref{eq1}, see for example \cite{DawsonGorostiza}. However, such topic is quite technical and hard to deal with. Thus, we have chosen to modify the definition of test functions in order to have \eqref{eq1}, so that the proof of uniqueness, presented in \cite{fgn2}, is valid. 
 
 Finally, some details missing in the original paper are also included here. 

\section{Introduction}
The purpose of these notes is to present a correction in the definition of the  space of test functions of the Ornstein-Uhlenbeck processes which govern the equilibrium density fluctuations of symmetric simple exclusion processes with a slow bond as defined in \cite{fgn2}, that we denote by $\{\eta_{tn^2}:t\geq 0\}$. We consider the processes starting from the invariant state, namely the Bernoulli product measure of parameter $\rho$. Fix $\rho\in (0,1)$. We recall now the notion of the density fluctuation field as the linear functional acting on test functions $H\in\mathcal{S}_\beta(\mathbb{R})$ as 
\begin{equation*}
\mathcal{Y}^n_t(H)=\frac{1}{\sqrt n}\sum_{x\in\mathbb{Z}}H\Big(\frac{x}{n}\Big)(\eta_{tn^2}-\rho)\,,
\end{equation*}
 see \cite[Subsection 2.5]{fgn2}. The limiting process  of $\mathcal{Y}^n_t$, denoted by  $\mathcal{Y}_t$,  is a generalized Ornstein-Uhlenbeck process solution of
\begin{equation}\label{eq Ou}
d\mathcal{Y}_t=\A \mathcal{Y}_tdt+\sqrt{2\chi(\rho)}\B d\mc{W}_t\,,
\end{equation}
as explained in \cite{fgn2}, where  $\chi(\rho)=\rho(1-\rho)$, the stochastic process $\mc{W}_t$ is an $\mc S'_\beta(\bb R)$-valued Brownian motion, and $\Delta_\beta$ and $\nabla_\beta$ are properly defined below.
\section{The space of test functions}
We start by redefining the space of test functions $\mc S_\beta(\bb R)$ where the operators 
$\Delta_\beta$ and $\nabla_\beta$ are defined.

\subsection{Definition of $\mathcal{S}_\beta(\bb R)$ and the operators $\Delta_\beta$ and $\nabla_\beta$}
\begin{definition}
For $\beta\in [0,+\infty]$ and $\alpha>0$, let $L^2_\beta(\mathbb{R})$ be the space of functions $H:\mathbb{R}\rightarrow{\mathbb{R}}$ with $\|H\|_{2,\beta}^2:=\|H\|_{2}^2+\frac{1}{\alpha^2}\textbf{1}_{\beta=1}(H(0))^2<+\infty$, where $\|H\|_{2}^2= \int_{\mathbb{R}}H(u)^2 du$.
\end{definition}

\begin{definition} Let $\mc S(\bb R \backslash \{0\})$ be the space of functions $H:\bb R\to\bb R$ such that: \begin{enumerate}
\item[(i)]  $H$ is smooth on $\bb R\backslash\{0\}$, i.e. $H \in \mc C^{\infty}(\bb R\backslash\{0\})$,
\item[(ii)]  $H$ is continuous from the right at 0,
\item[(iii)]  for all non-negative integers $k,\ell$, the function $H$ satisfies
\begin{equation}\label{eq2.2}
\Vert H \Vert_{k,\ell}:= \sup_{u \neq 0} \Big| (1+|u|^\ell) \frac{d^kH}{du^k}(u)\Big| < \infty.
\end{equation}
\end{enumerate}
\end{definition}
\begin{remark}
It is a consequence of \eqref{eq2.2} that  the side limits $\frac{d^kH}{du^k}(0^+)$ and  $\frac{d^kH}{du^k}(0^-)$
exist for any integer $k \ge 0$.
\end{remark}
\begin{definition}\label{def23}  
\begin{enumerate}
\quad

\item For $\beta <1$, $\mc S_\beta(\bb R)$ is the usual Schwartz space $\mc S(\bb R)$.
 
\item Fix $\alpha>0$. For $\beta=1$, $\mc S_\beta(\bb R)$ is the subset of $\mc S(\bb R \backslash \{0\})$ composed of functions $H$ such that for any integer $k\geq 0$
\begin{equation} \frac{d^{2k+1}H}{du^{2k+1}}(0^+)= \frac{d^{2k+1}H}{du^{2k+1}}(0^-) = \alpha  \Big(\frac{d^{2k}H}{du^{2k}}(0^+)-\frac{d^{2k}H}{du^{2k}}(0^-)\Big)
\label{eq:bound_beta_1}\end{equation}

\item For $\beta >1$, $\mc S_\beta(\bb R)$ is the subset of $\mc S(\bb R \backslash \{0\})$ composed of functions $H$ such that for any integer $k\geq 0$ \begin{equation}\frac{d^{2k+1}H}{du^{2k+1}}(0^+)=\frac{d^{2k+1}H}{du^{2k+1}}(0^-)=0.
\label{eq:bound_beta_ge_1}\end{equation}
\end{enumerate}
Finally, let $\mc S'_\beta(\bb R)$ be the topological dual of $\mc S_\beta(\bb R)$. 
\end{definition}
We notice that the spaces defined above are very close to those of \cite{fgn2}, but here we also impose boundary conditions on higher order derivatives. Since the results in \cite{fgn2} only use the boundary conditions for zero and first order derivatives, all the results there remain  true for the spaces $S_\beta(\mathbb{R})$  introduced above. As in \cite{fgn2} the spaces $S_\beta(\mathbb{R})$  are  Fr\'echet space and this fact is only used when showing tightness of the processes $\{\mathcal{Y}_t^n:t\in[0,T]\}_{n\in\mathbb{N}}$, see \cite{Mitoma}.

\begin{definition}\label{def24}
Fix $\beta\in[0,+\infty]$. 
Let $\nabla_\beta$ and $\Delta_\beta$ be the operators acting on functions $H \in \mc S_\beta(\bb R)$ as
$$
\nabla_\beta H(u)= \frac{dH}{du}(u)\textbf{1}_{u \neq 0}+  \frac{dH}{du}(0^+)\textbf{1}_{u = 0}\; \textrm{ and}\;  \Delta_\beta H(u)=  \frac{d^2H}{du^2}(u)\textbf{1}_{ u \neq 0}+\frac{d^2H}{du^2}(0^+)\textbf{1}_{u = 0}.
$$
\end{definition}

 We recall now the hydrodynamic equations (and their semigroups) associated to the different regimes of $\beta$.\bigskip

\noindent 
\textbf{The regime $\beta\in[0,1)$.}
The hydrodynamic equation  is the heat equation on the line given by
\begin{equation}\label{pde1}
%\left\{
%\begin{array}{ll}
\partial_t u(t,x) = \; \partial_{xx}^2 u(t,x),
% &t %\geq 0,\, x \in \mathbb R\,.\\
%u(0,x) = \; g(x), &x \in \mathbb R.
%\end{array}
%\right.
\end{equation}
Let $\phi_t(x)=\frac{1}{\sqrt{4\pi t}}e^{-\frac{x^2}{4t}}$ be the heat kernel.
It is a classical fact that the semigroup related to \eqref{pde1} is given by
\begin{equation}\label{sem heat eq}
  T^\beta_t g(x):=\phi_t\star g(x)= \frac{1}{\sqrt{4\pi t}}\int_{\bb R} e^{-\frac{(x-y)^2}{4t}}g(y)\,dy\,,\quad \textrm{for }x\in\bb R\,,
 \end{equation}
 where $\star$ is the convolution operator.
In this case we denote  $T^\beta_t$ simply by  $T_t$. 

\bigskip

\noindent 
\textbf{The regime $\beta\in(1,+\infty]$.}
In this case, the hydrodynamic equation  is the heat equation with Neumann's boundary conditions at $x=0$ given by
\begin{equation}\label{pde3}
\left\{
\begin{array}{ll}
\partial_t u(t,x) = \; \partial^2_{xx} u(t,x), &t \geq 0,\, x \in \mathbb R\backslash\{0\},\\
\p_x u(t,0^+)=\p_x u(t,0^-)=0,  &t \geq 0\,.\\
%u(0,x) = \; g(x), &x \in \mathbb R.
\end{array}
\right.
\end{equation}
Its semigroup is given by
  \begin{equation}\label{sem heat eq neu}
  T_t^\beta g(x):=\int_{0}^{+\infty}\Big[\phi_t(x-y)+\phi_t(x+y)\Big]g(y\; \textrm{sign}(x))\,dy.
 \end{equation}
In this case we denote $T^\beta_t$ as $T_t^\textrm{Neu}$. 
\bigskip

\noindent
\textbf{The regime $\beta=1$.} Let $\alpha>0$.
In this case, the hydrodynamic equation is the heat equation with Robin's boundary condition  at $x=0$ given by
\begin{equation}\label{pde2}
\left\{
\begin{array}{ll}
 \partial_t u(t,x) \;= \; \partial^2_{xx} u(t,x), &t \geq 0,\, x \in \mathbb R\backslash\{0\},\\
\p_x u(t,0^+)=\p_x u(t,0^-)\;=\alpha\{u(t,0^+)-u(t,0^-) \},  &t \geq 0.\\
 %u(0,x)\; = \; g(x), &x \in \mathbb R.
\end{array}
\right.
\end{equation}
Denote by $g_{\textrm{even}}$ (resp. $g_{\textrm{odd}}$) the even (resp. odd) parts of a function $g:\bb R\to \bb R$: for $x\in{\mathbb{R}}$,
$g_{\textrm{even}}(x)=\frac{g(x)+g(-x)}{2}$ and $g_{\textrm{odd}}(x)=\frac{g(x)-g(-x)}{2}\,.$
 The semigroup associated to \eqref{pde2} has been obtained in \cite{fgn2} by symmetry arguments and its expression is given by
  \begin{equation*}
  \begin{split}
  T_t^\beta g(x)= &\int_{\bb R}
\phi_t(x-y) g_{\textrm{{\rm even}}}(y)\,dy 
    + \textrm{sign}(x)e^{2\alpha |x|}\\\times &\int_{|x|}^{+\infty} e^{-2\alpha z} \int_0^{+\infty}
\Big[(\pfrac{z-y+4\alpha t}{2t})\phi_t(z-y)+(\pfrac{z+y-4\alpha t}{2t})\phi_t(z+y)\Big]
g_{\textrm{{\rm odd}}}(y)\, dy\, dz.
\end{split}
  \end{equation*}
 We denote $T^\beta_t$ here by $T^\alpha_t$.

Now we state the central result of this paper. 
 \begin{proposition}\label{main}
For any $\beta\in[0,+\infty]$, for all $t>0$ and for all $H\in \mc S_\beta(\bb R)$:  
 \begin{align}
 & \Delta_\beta  H\in \mc S_\beta(\bb R),\label{inc1}\\
& T_t^\beta H\in \mc S_\beta(\bb R).\label{inc2}
\end{align}
\end{proposition}
 Before proving this proposition, we make a break to explain where each one of the two conditions above are used in \cite{fgn2}.\label{page}

In the middle of page 4170 of \cite{fgn2}, we have that, for $H\in \mc S_\beta(\bb R)$,
\begin{equation}\label{eq212}
\mc M_t(H)\;:=\;\mc Y_t(H)-\mc Y_0(H)-\int_0^t \mc Y_s(\Delta_\beta H)\,ds\,,
\end{equation}
is a martingale.  Since $\mc Y_t$ is a linear functional defined on $\mc S_\beta(\bb R)$,  in order to have $\mc M_t(H)$ well defined,  we need \eqref{inc1} to be true.

On the other hand, at the bottom of page 4175 of \cite{fgn2},  for $H\in\mc{S}(\bb R)$,
\begin{equation}\label{eq213}
Z_t(H)\;:=\;\exp\Big\{\frac{1}{2}\int_0^t \Vert \nabla_\beta T^\beta_{S-r} H\Vert^2_{2,\beta}\,dr+i \,\mc Y_t(T^\beta_{S-t}H)\Big\}
\end{equation}
is a martingale. Therefore, we need to be \eqref{inc2} in order to have $Z_t(H)$ well defined.

We note that we make use of the martingale \eqref{eq212} to prove the existence of the generalized Ornstein-Uhlenbeck process solution of \eqref{eq Ou} described in \cite{fgn2}, while \eqref{eq213} is needed in the proof of uniqueness of its solutions.

 \begin{proof}[Proof of Proposition \ref{main}]
 Let $H\in \mc S_\beta(\bb R)$. For any $\beta$, one can check \eqref{inc1} by easily applying the definition of $\mc S_\beta(\bb R)$. Hence, it remains only to show \eqref{inc2}, whose proof we split according to each regime of $\beta$.
 
We observe that we prove \eqref{inc2} under the assumption $H\in \mc S_\beta(\bb R)$. Nevertheless, it  remains true under the weaker assumption $H\in L^1(\bb R)$.

 \textbf{$\bullet$ $\beta\in[0,1)$.} This case is straightforward since it corresponds to the classical situation where $\Delta_\beta$ is the usual Laplacian, $\mc S_\beta(\bb R)$ is the Schwartz space and $T_t^\beta$ is the heat semigroup.

For the next two cases, keep in mind that the derivative of a smooth even real function on the line is an odd function and that the derivative of a smooth odd real function on the line is an even function.\medskip

 \textbf{$\bullet$  $\beta\in(1,+ \infty]$.} 
We first claim that $\Delta_\beta T_t^{\textrm{Neu}} H$ is again a solution of \eqref{pde3} with initial condition $\Delta_\beta H$. Provided by this claim and doing an induction procedure (on the derivatives of even order of $T_t^{\textrm{Neu}} H$)  we are lead to  \eqref{eq:bound_beta_ge_1} for $T_t^{\textrm{Neu}} H$, from which  \eqref{inc2} will follow.

 One way to prove the claim is to check it directly by differentiating \eqref{sem heat eq neu} twice with respect to space. Alternatively,  a more elegant proof can be done by recalling how \eqref{sem heat eq neu} is usually deduced in the literature and that is what we do now. Keep in mind that the derivatives ahead are in the classical sense. First of all, since $\p_t  T_t^{\textrm{Neu}} H = \partial^2_{xx} T_t^{\textrm{Neu}} H$, for any $x>0$ and $t>0$, differentiating twice in space we get $\p_t  \partial^2_{xx} T_t^{\textrm{Neu}} H = \partial^2_{xx} \big(\partial^2_{xx} T_t^{\textrm{Neu}} H\big)$. Hence it only remains to show that $\partial^2_{xx} T_t^{\textrm{Neu}} H$ satisfies the  correct boundary condition at zero. 
 
 We make a break to explain how  the expression for  $T_t^{\textrm{Neu}} H$ is usually deduced. In the positive half-line, one has to restrict the initial profile $H$ to the positive half-line and extend it to an even function in the whole line by taking $H(x)=H(-x)$ for $x<0$. Next, we evolve this even function according to \eqref{sem heat eq}, the semigroup of the heat equation in $\bb R$. Since the semigroup \eqref{sem heat eq} preserves  even functions, and a smooth even function has zero derivative at zero, we conclude that \eqref{sem heat eq neu} is the solution of \eqref{pde3} in the positive half-line. The same argument applies to the negative half-line. In other others words,  this says that  $ T_t^{\textrm{Neu}}H$ given in \eqref{sem heat eq neu} is the solution of \eqref{pde3} with initial condition $H$. 

 Observe now that an even smooth function has its third order derivative  at zero   vanishing. Since $T_t^{\textrm{Neu}}H$ was constructed above as  a restriction of an even function in each half-line, this implies that $(\p_x \partial^2_{xx} T_t^{\textrm{Neu}}H)(0^\pm)=0$, proving the claim.

 \textbf{$\bullet$  $\beta=1$.} The scheme of proof in this case is similar to the previous one. 
  Here, the important property of  $T^\alpha_t$ is its symmetry which is also used for its deduction.  First,  decompose the initial condition $g$ in its odd and even parts. By a symmetry argument, one can figure out that
\begin{equation}\label{eq34}
 T_t^\alpha g(x)\;=\;
  T_tg_{\textrm{even}}(x)+\textrm{sign}(x)\tilde{T}_t^\alpha g_{\textrm{odd}}(|x|)
\end{equation}
where $\tilde{T}_t^\alpha$ is the semigroup of the following partial differential equation in the half-line:
\begin{equation}\label{pde4}
\left\{
\begin{array}{ll}
\partial_t u(t,x) = \; \partial^2_{xx} u(t,x), &t \geq 0,\, x >0,\\
\p_x u(t,0^+)= 2\alpha u(t,0^+),  &t \geq 0.\\
%u(0,x) = \; g(x), &x >0.
\end{array}
\right.
\end{equation}

Let $H\in \mc S_\beta(\bb R)$ and recall that $ T^\alpha_t H$ is the solution of \eqref{pde2} with initial condition $H$. We claim that $\Delta_\beta T^\alpha_t H$ is  a solution of \eqref{pde2} with initial condition $\Delta_\beta H$.  Provided by this claim, analogously to what we did before, we conclude that  $T^\alpha_t H$ satisfies the boundary conditions of \eqref{eq:bound_beta_1}, which implies \eqref{inc2}. By   \eqref{eq34} and due to the fact that an even smooth function has zero derivative at  zero, in order to prove the claim it is enough to show that $\partial^2_{xx} \tilde{T}^\alpha_t H_{\rm odd}$ is again a solution of \eqref{pde4} with initial condition $\partial^2_{xx} H_{\rm odd}$.  In other words, we have  to assure that differentiating twice (in space) a solution of \eqref{pde4} we obtain again a solution of \eqref{pde4} with the same boundary condition (but with different initial condition).
Denote by $u$ the solution of \eqref{pde4} with initial condition $H\in\mc S_\beta(\bb R)$ and let
\begin{equation}\label{A8}
v=2\alpha u-\p_x u\,,
\end{equation}
which   is the solution of the following equation
\begin{equation}\label{eq Dir}
\left\{
\begin{array}{ll}
\partial_t v(t,x) = \; \partial^2_{xx} v(t,x), &t \geq 0,\, x >0,\\
v(t,0^+)= 0,  &t \geq 0,\\
\end{array}
\right.
\end{equation}
with initial condition $v_0=2\alpha H-\p_x H$. Last equation is the heat
equation with  Dirichlet boundary conditions at $x=0^+$.  The semigroup  associated to last equation, that we denote by $T^{\textrm{Dir}}_t$, is
given by
\begin{equation}\label{sem dir}
 T^{\textrm{Dir}}_tv_0(x):=\int_0^{{+\infty}}\Big[\phi_t(x-y)-\phi_t(x+y)\Big]v_0(y)\,dy\,,
\end{equation}
where $\phi_t$ is the heat kernel.
If we show that $\partial^2_{xx} T^{\textrm{Dir}}_t v_0$ is again a solution of \eqref{eq Dir}, then by \eqref{A8} we conclude that $\partial^2_{xx} \tilde{T}^\alpha_t H_{\rm odd}$ is again a solution of \eqref{pde2}. 

To  see that  \eqref{sem dir} is the solution of \eqref{eq Dir}, we perform a symmetry argument similar to the one used for the heat equation with Neumann boundary conditions. More precisely, 
 given an initial condition $\nu_0$, we first restrict it to the positive half-line, then we extend it to an odd function in the entire real line. After that, we evolve this function accordingly to \eqref{sem heat eq}. Since \eqref{sem heat eq} preserves  odd functions and any  smooth odd function 
vanishes at the origin, we conclude that \eqref{sem dir} is the solution of \eqref{eq Dir}. 

We point out that the second derivative at zero of a smooth odd function vanishes. Therefore  $\partial^2_{xx} T^{\textrm{Dir}}_t v_0$ is  a solution of \eqref{eq Dir} with initial condition $\partial^2_{xx} v_0$, proving the claim and concluding the proof.
 \end{proof}

\section{Uniqueness of the Ornstein-Uhlenbeck processes}
In this section we give more details on four intermediate results that we need in  the proof presented in  page 4176 of \cite{fgn2} for the uniqueness of solutions of the Ornstein-Uhlenbeck process given by \eqref{eq Ou}.
These four results are related to the continuity of  $T_t^{\beta}$.

\begin{proposition}\label{Prop32}
Fix $\beta\in [0,+ \infty]$ and let $T_t^\beta:\mc S_\beta(\bb R)\to \mc S_\beta(\bb R)$ be  as before. Then, for each $H\in \mc S_\beta(\bb R)$, the function $f:\;\bb R_+\to \mc S_\beta(\bb R)$ defined by $f(t)=T^\beta_t H$ is uniformly continuous with respect to the topology of $\mc S_\beta(\bb R)$. 
\end{proposition}
Above we mean that the convergence is with respect to all   norms $\Vert \cdot \Vert_{k,\ell}$  defined in \eqref{eq2.2}. Such topology can be induced by the metric given on $H,G\in \mc S_\beta(\bb R)$ by
\begin{equation}\label{eq32a}
d(H,G)\;:=\;\sum_{k,\ell=1}^\infty \frac{1}{2^{k+\ell}}\Vert H-G\Vert_{k,\ell}\,.
\end{equation}

\begin{proof}[Proof of Proposition \ref{Prop32}]
First of all, by \eqref{eq32a} we note that it is sufficient to show that $f$ is uniformly continuous for each one of the norms $\Vert \cdot \Vert_{k,\ell}$.

We start with the regime  $\beta\in [0,1)$. In this case, since  $T_t H= \phi_t \star H$,	by the properties of the space Fourier transform  $\mc F$  described in, for example,  Chapter IX of \cite{reedsimon}, we get 
\[
\begin{split}
T_t H & \;=\; \mc F^{-1} \mc F \big[\phi_t \star H\big]\;=\; \sqrt{2\pi}\; \mc F^{-1} \big[(\mc F \phi_t) \cdot(\mc F H)\big] \;=\;\sqrt{2\pi}\; \mc F^{-1} \Big[
\frac{e^{-\lambda^2 t}}{\sqrt{2\pi}} \cdot (\mc F H)\Big]\\
& \;=\;\; \mc F^{-1} \Big[
e^{-\lambda^2 t} \cdot(\mc F H)\Big]\,,\\
\end{split}
\] 
which implies that $f$ is uniformly continuous in each one of  the norms $\Vert \cdot\Vert_{k,\ell}$.
%Differentiating under the sign of integral (Leibniz Theorem) we have that
%\begin{equation*}
%\frac{\p^{n} T_t^\beta H}{\p x^n}(x)\;=\; \frac{1}{\sqrt{4\pi t}}\int_{\bb R} H(y)\, \frac{\p^n}{\p x^n}\Big(e^{-\frac{(x-y)^2}{4t}}\Big)\,dy\,,\quad \textrm{for }x\in\bb R\,.
%\end{equation*}
%An induction in $n\in \bb N$  shows that 
%\begin{equation*}
%\frac{\p^{n} T_t^\beta H}{\p x^n}(x)\;=\; \frac{1}{\sqrt{4\pi t}}\int_{\bb R} H(y)\, e^{-\frac{(x-y)^2}{4t}}\,p_n\big(\pfrac{x-y}{4t}\big)\,dy\,,\quad \textrm{for }x\in\bb R\,.
%\end{equation*}
%where $p_n$ is a polynomial of degree $n$.
%Keeping this last formula in mind, it is easy to conclude that the expression above  varies continuously  (as a function of $t\in \bb R_+$) in the supremum norm for any $n\in \bb N$.  Moreover, this continuity remains in force if we multiply the expression above by a polynomial on $x$.   In others words, it means that the function $f$ defined in \eqref{eq311} is continuous. 

Last argument can also be used for the other cases $\beta=1$ and $\beta\in(1,+\infty]$, since for those cases the semigroups are written in terms of the semigroup \eqref{sem heat eq}, as one can see in \eqref{sem heat eq neu}, \eqref{eq34} and \eqref{sem dir}. 
\end{proof}
\begin{proposition}\label{Prop31}
Fix $\tau>0$ and $\beta\in [0,+\infty]$. For any $H\in \mc S_\beta(\bb R)$ and $\eps>0$, we have that:
$T^\beta_{t+\eps}H - T^\beta_{t}H \;=\; \eps\, \Delta_\beta T^\beta_{t}H +  o (\eps,t).$
For each $\eps>0$,  $o(\eps,t)$ denotes a function in $S_\beta(\bb R)$ such that
$\lim_{\eps \searrow 0}  \frac{o(\eps,t)}{\eps} \;=\;0
$
holds in the topology of $S_\beta(\bb R)$, being the limit uniform in $t\in [0,\tau]$. 
\end{proposition}
\begin{proof}

We start with the regime  $\beta\in[0,1)$.  Let $H\in \mc S_\beta(\bb R)$ and consider $t>0$.
Recall that  $T_t$ is given by \eqref{sem heat eq}. Differentiating under the sign of the integral (Leibniz Theorem) we have, for $x\in\mathbb{R}$, that
\begin{equation*}
\frac{\p^{n} T_tH}{\p x^n}(x)\;=\; \int_{\bb R} H(y)\, \frac{\p^n \phi_t}{\p x^n}(x-y)\,dy.
\end{equation*}
An induction in $n\in \bb N$  easily shows that 
\begin{equation}\label{polynomial}
\frac{\p^{n} T_t H}{\p x^n}(x)\;=\; \int_{\bb R} H(y)\, \phi_t(x-y)\,p_n\big(\pfrac{x-y}{4t}\big)\,dy\,,\quad \textrm{for }x\in\bb R\,,
\end{equation}
where $p_n$ is a polynomial of degree $n$.  On the other hand, it can be checked directly that, as $\eps\searrow 0$,
$\frac{1}{\eps}\Big\{\phi_{t+\eps}(x-y)-\phi_t(x-y)\Big\}$
converges uniformly in t to $\frac{\p^2}{\p x^2}\phi_t(x-y).$
More than that, because of the fast decay at $x=\pm\infty$ of  $\phi_t(x-y)$ together with  \eqref{polynomial}, we have that, for any $n\in \bb N$ and any polynomial $p_k(x)$ of degree $k$, 
\begin{equation*}
\frac{p_k(x)}{\eps}\,\frac{\p^n}{\p x^n}\Big(\phi_{t+\eps}(x-y)-\phi_t(x-y)\Big)
\end{equation*}
converges, uniformly in $t$, to
$p_k(x)\,\frac{\p^{n+2}}{\p x^{n+2}}\phi_t(x-y)
$, as $\eps\to0$.
The uniform convergence in time stated above can be simply  rephrased as
\begin{equation}\label{eq3.3}
\phi_{t+\eps}(x-y)-\phi_t(x-y)\;=\; \eps\, \Delta_\beta \phi_t (x-y)+o(\eps,t). 
\end{equation}
  Finally,  multiplying the equation above by a function $H\in \mc S_\beta(\bb R)$ and integrating in $y\in \bb R$ we get that
 \begin{equation}\label{eq3.4}
T^\beta_{t+\eps}H - T^\beta_{t}H \;=\; \eps\, \Delta_\beta T^\beta_{t}H +  o (\eps,t)\,,
\end{equation}
as desired. We point out that for $t=0$, \eqref{eq3.3} does not make sense. Nevertheless, it is classical that as $t\searrow 0$, $\phi_t(x-y)$  converges, in the sense of distributions, to the Dirac measure at the point $y$,  \cite{reedsimon}. Moreover, the operator $\Delta_\beta$ is continuous and its semigroup $T_t^\beta$ is uniformly continuous on $t$, with respect to the topology of $\mc S_\beta(\bb R)$. 
Therefore,   \eqref{eq3.4} still makes sense for $t=0$ and one  concludes  that the convergence $\lim_{\eps\searrow 0}\frac{o(\eps,t)}{\eps}=0$ is uniform for $t\in [0,\tau]$.
By the same reason as above this also proves the result for the other cases $\beta=1$ and $\beta\in(1,+\infty]$.
\end{proof}

\begin{proposition} Fix  $\beta\in [0,+\infty]$ and recall that $\mc Y_\cdot$ is a solution of \eqref{eq Ou}. Then, for any $H\in\mathcal{S}_\beta(\bb R)$  the function $f:\;\bb R_+\times \bb R_+ \to \bb R$ defined by $f(s,t)= \mc Y_t(T^\beta_s H)$
is  continuous.
\end{proposition}
\begin{proof}
This follows from  Proposition \ref{Prop32} and  the fact that $\mc Y_\cdot\in\mathcal{C}([0,T], \mc S_\beta'(\bb R))$.
\end{proof}

\begin{proposition} Fix $\beta\in[0,+\infty]$. For any $H\in \mc S_\beta(\bb R)$, the function $f:\;\bb R_+\to \bb R$ defined by $f(t)= \Vert \nabla_\beta  T^\beta_t H\Vert^2_{2,\beta}$
is continuous.
\end{proposition}
\begin{proof}
Let $g(x)=1+ x^2$ and $s,t\geq 0$.
Recall the inequality $\big| \Vert a\Vert -\Vert b\Vert \big|\leq \Vert a-b\Vert$, which is true for any norm $\Vert \cdot\Vert$. Applying this result to the norm $\Vert \cdot \Vert_{2,\beta}^2$, we get
\[
\begin{split}
|f(t)-f(s)|& \;= \;\Big|
\Vert \nabla_\beta  T^\beta_t H \Vert^2_{2,\beta} - \Vert \nabla_\beta  T^\beta_s H\Vert^2_{2,\beta}\Big| \\
& \;\leq \;
\Vert \nabla_\beta  T^\beta_t H - \nabla_\beta  T^\beta_s H\Vert^2_{2,\beta} \\
&\;=\;   \Big\Vert \frac{1}{g}  \cdot g \Big(\nabla_\beta  T_t^\beta H - \nabla_\beta  T_s^\beta H\Big)\Big\Vert^2_{2,\beta} \\
& \;\leq \; \Vert 1/g\Vert_{2,\beta} ^2 \;\cdot\;\sup_{x\in \bb R} \Big\vert (1+x^2) \Big(\nabla_\beta  T_t^\beta H(x) - \nabla_\beta  T_s^\beta H(x)\Big) \Big\vert\\
& \;\leq \; \Vert 1/g\Vert_{2,\beta} ^2\;\cdot\;\Big\{\Vert   T_t^\beta H -  T_s^\beta H \Vert_{1,0} + \Vert   T_t^\beta H -  T_s^\beta H\Vert_{1,2} \Big\}\,,
\end{split}
\] 
according to \eqref{eq2.2}. Since $\Vert 1/g\Vert_{2,\beta}^2<+\infty$ and recalling Proposition \ref{Prop32}, we conclude that $f$ is a continuous function.
\end{proof}

\section*{Acknowledgements}
The authors thank to C\'edric Bernardin for pointing out  the issues concerning this paper and also for valuable discussions on the subject.

This work was partially supported by FCT/Portugal through the project UID/MAT/04459/2013. 
A. Neumann is partially supported by FAPERGS (proc. 002063-2551/13-0) and T. Franco is partially supported by FAPESB (Jovem Cientista 9922/2015).

\bibliographystyle{plain}
\bibliography{bibliography}

%\bibliographystyle{amsplain}
%
%\begin{thebibliography}{}
%
%
%\bibitem{dg}D. Dawson, L. Gorostiza. \emph{Generalized solutions of a class of nuclear-space-valued stochastic evolution equations}. 
%Applied Mathematics and Optimization,
% Volume 22, Issue 1, pages 241-263, (1990).
%
%     
%
%
%
%
% \bibitem{fgn3}
%T. Franco, P. Gon\c calves, A. Neumann. {\em Phase transition in equilibrium fluctuations of symmetric slowed exclusion}.
%Stochastic Processes and their Applications, Volume 123, Issue 12,  Pages 4156--4185, (2013).
%
%
%
%
%\bibitem{Mit}
%I. Mitoma, {\em{ Tightness of probabilities on {$C([0,1];{\mathcal
%  S}\sp{\prime} )$} and {$D([0,1];{\mathcal S}\sp{\prime} )$}}}. Ann. Prob.,  11, no. 4, 989--999, (1983).
%
%
%\bibitem{rs}
%M. Reed and B. Simon. {\em Functional Analysis,
%Volume 1 (Methods of Modern Mathematical Physics)}. Academic Press, 1
%edition, (1981).
%
%
%
%
%
%
%\end{thebibliography}
%

\end{document}